\documentclass[leqno,12pt]{amsart} 
\usepackage{amsmath}    
\usepackage{amssymb}  
\usepackage{amscd}
\usepackage{amsthm}  
\usepackage{graphicx}
\usepackage{comment}	
\usepackage{txfonts} 	
\usepackage[dvipdfm,
  colorlinks=false,
  bookmarks=true,
  bookmarksnumbered=false,
  pdfborder={0 0 0},
  bookmarkstype=toc]{hyperref} 

\numberwithin{equation}{section}
\theoremstyle{plain} 
\newtheorem{thm}{\indent\sc Theorem}[section] 

\newtheorem{conj}[thm]{\indent\sc Conjecture}

\theoremstyle{definition}
\newtheorem{dfn}[thm]{\indent\sc Definition}
\newtheorem{exa[thm]}{\indent\sc Example}

\newtheorem{rem}[thm]{\indent\sc Remark}

\newcommand{\hol}[2] { #1 \rightarrow #2 }
\newcommand{\kah}{K\"{a}hler }

\newcommand{\deldel}{\sqrt{-1}\partial \overline{\partial}}

\begin{document}
\baselineskip=16pt

\title[On the global generation of 
direct images of pluri-adjoint line bundles]
{ On the global generation of 
direct images of pluri-adjoint line bundles}
\author{Masataka Iwai}

\address{Graduate School of Mathematical Sciences, The University of Tokyo, 3-8-1 Komaba,
Tokyo, 153-8914, Japan.}

 \email{{\tt
masataka@ms.u-tokyo.ac.jp}}

\maketitle
\begin{abstract}
We study the Fujita-type conjecture proposed by Popa and Schnell. 
We obtain an effective bound on the global generation
of direct images of pluri-adjoint line bundles
on the regular locus.
We also obtain an effective bound on 
the generic global generation
for a Kawamata log canonical $\mathbb{Q}$-pair.
We use analytic methods such as $L^2$ estimates, $L^2$ extensions and
injective theorems of cohomology groups.
\end{abstract}

\section{Introduction}
The aim of this paper is to give a partial answer to the following conjecture by Popa and
Schnell.
This conjecture is a version of Fujita's conjecture.

\begin{conj} [\cite{PS14} Conjecture 1.3]
\label{conjp}
Let $f \colon \hol{X}{Y} $ be a surjective morphism of smooth projective varieties,
with $Y$ of dimension $n$, and $L$ be an ample line bundle on $Y$.
For any $a \geq 1$, the sheaf 
$$ f_{*}( K_{X}^{\otimes a} ) \otimes L^{\otimes b} 
$$  
is globally generated for all $b \geq a(n+1)$.

\end{conj}
In \cite{PS14}, Popa and Schnell proved this conjecture in the case when $L$ is 
 ample and globally generated.
After that,
Dutta removed the global generation assumption on $L$ making 
a statement about generic global generation. 
\begin{thm}[\cite{Dutta17} Theorem A]
\label{Dutta}
Let $(X , \Delta )$ be a Kawamata log canonical $\mathbb{Q}$-pair of a normal projective variety and an effective divisor, 
and $Y$ be a smooth projective $n$-dimensional variety.
Let $f \colon \hol{X}{Y} $ be a surjective morphism,
and $L$ be an ample line bundle on $Y$.
For any $a \geq 1$ such that $a(K_{X} + \Delta ) $ is an integral Cartier divisor, the sheaf 
$$ 
f_{*}\Bigl( \mathcal{O}_X \bigl( a(K_{X} + \Delta ) \bigr) \Bigr) \otimes L^{\otimes b} 
$$  
is generated by the global sections at a general point $y \in Y$ either
\begin{enumerate}
\item for all $b \geq a \bigl( \frac{n(n + 1)}{2} +1 \bigr) $, or
\item for all $b \geq a(n + 1)$ when $n \le 4$.
\end{enumerate}
\end{thm}

On the other hand, Deng obtained a linear bound for $b$ 
when $a$ is large by using 
analytic methods.

\begin{thm}[\cite{Deng17} Theorem C]
With the above notation and in the setting of Theorem \ref{Dutta},
for any $a \geq 1$ such that $a(K_{X} + \Delta ) $ is an integral Cartier divisor, the sheaf 
$$ 
f_{*}\Bigl( \mathcal{O}_X \bigl( a(K_{X} + \Delta ) \bigr) \Bigr) \otimes L^{\otimes b} 
$$  
is generated by the global sections at a general point $y \in Y$ either
\begin{enumerate}
\item for all $b \geq n^2 - n + a(n+1) $, or
\item for all $b \geq n^2  + 2$ when $K_Y$ is pseudo-effective.
\end{enumerate}
\end{thm}

Now we state our results.
First, we treat the case when $X $ is smooth and  $\Delta = 0$.
In \cite{Dutta17}, Dutta proved that
if $K_{X}^{\otimes a}$ is relatively free on the regular locus of $f$, 
 $ f_{*}( K_{X}^{\otimes a} ) \otimes L^{\otimes b} $ is 
generated by the global sections at any regular value of $f$ for all $b \geq a \bigl( \frac{n(n + 1)}{2} +1 \bigr) $
.
In this paper, we can remove this assumption and obtain a better bound for $b$.

\begin{thm}
\label{main}
Let $f \colon \hol{X}{Y} $ be a surjective morphism of smooth projective varieties,
with $Y$ of dimension $n$, and $L$ be an ample line bundle on $Y$.
If $y$ is a regular value of $f$, then for any $a \geq 1$ the sheaf 
$$ f_{*}(K_{X}^{\otimes a} ) \otimes L^{\otimes b} 
$$  
is generated by the global sections at $y$ for all $b \geq \frac{n(n-1)}{2} + a(n+1)$.
\end{thm}

In this paper, $X_y$ is a smooth connected variety  for any regular value $y \in Y$.
In particular, if $f$ is smooth, $ f_{*}(K_{X}^{\otimes a} ) \otimes L^{\otimes b} $
is globally generated for all $b \geq \frac{n(n-1)}{2} + a(n+1)$.
We give a partial answer to Conjecture \ref{conjp}.

Second, we treat a log case.
In this case, we obtain the same bound as Theorem \ref{main} about generic global generation
even when $X$ is a complex analytic variety.

\begin{thm}
\label{main2}
Let $(X , \Delta) $ be a Kawamata log terminal $\mathbb{Q}$-pair of a normal complex analytic variety in Fujiki's class $\mathcal{C}$ and an effective divisor, 
and $Y$ be a smooth projective $n$-dimensional variety.
Let $f \colon \hol{X}{Y} $ be a surjective morphism,
and $L$ be an ample line bundle on $Y$.
For any $a \geq 1$ such that $a(K_{X} + \Delta ) $ is an integral Cartier divisor, the sheaf 
$$ 
f_{*}\Bigl( \mathcal{O}_X \bigl( a(K_{X} + \Delta ) \bigr) \Bigr) \otimes L^{\otimes b} 
$$  
is generated by the global sections at a general point $y \in Y$ either 

\begin{enumerate}
\item for all $b \geq \frac{n(n-1)}{2} + a(n+1) $
, or
\item for all $b \geq \frac{n(n-1)}{2} + 2$ when $K_Y$ is pseudo-effective.
\end{enumerate}
\end{thm}
\begin{rem}
After the author submitted this paper to arXiv,
Dutta told the author that she and Murayama obtained the same bounds as in Theorem \ref{main2} (1) in 
\cite[Theorem B]{DM}  by using the algebraic geometric methods  when $X$ is a normal projective variety.
Also, in \cite[Theorem B]{DM}, they obtained the linear bound when $(X , \Delta) $ is a log canonical 
$\mathbb{Q}$-pair.
For more details, we refer the reader to \cite{DM}.
\end{rem}

The organization of the paper is as follows. In Section 2, we review
some of the standard facts on an effective freeness of adjoint bundles
and a positivity of a relative canonical line bundle.
In Section 3, we prove Theorem \ref{main}.
In Section 4, we prove Theorem \ref{main2} by using the methods of Section 3 and \cite{Deng17}.

{\bf Acknowledgment. } 
The author would like to thank his supervisor Prof. Shigeharu Takayama for helpful comments and enormous support.
He wishes to thank Genki Hosono for several useful discussions.
He also wishes to thank Yajnaseni Dutta for letting him know their results
in \cite{DM} and sending 
their preprint.
He is greatly indebted to the referee for many kind advices.
This work is supported by the Program for Leading Graduate Schools, MEXT, Japan. This work is also supported by JSPS KAKENHI Grant Number 17J04457.

\section{Preliminary}

\begin{dfn}\cite[Chapter 13.]{Dem}
A function $\varphi$ : $\hol{X}{ [ -\infty , +\infty) } $ on a complex manifold $X$ is said to
be { \it quasi-plurisubharmonic } ({\it quasi-psh}) if $\varphi$ is locally the sum of a 
psh function and of a smooth function.
In addition, we say that $\varphi$ has { \it neat analytic singularities }
if for any $x \in X$ there exists an open neighborhood $U$ of $x$ such that 
$\varphi$ can be written
$$
\varphi = c \log \sum_{ 1 \le k \le L}| g_k(z) |^2 + w(z)
$$
on $U$ where $c\geq0$, $g_k \in \mathcal{O}_X(U)$ for any $k = 1 , \dots , L$ and $w \in \mathcal{C}^{\infty}(U)$.
\end{dfn}

\begin{dfn}\cite[Chapter 5.]{Dem}
If $\varphi$ is a quasi-psh function on a complex manifold $X$,
the multiplier ideal sheaf $\mathcal{J}( e^{- \varphi}) $ is a coherent
subsheaf of $\mathcal{O}_X$ defined by
$$
\mathcal{J}( e^{- \varphi})_x \coloneqq
\{ f \in \mathcal{O}_{ X , x} ; \exists U \ni x , \int_{U} | f |^2 e^{ - \varphi} d \lambda < \infty \},
$$
where $U$ is an open coordinate neighborhood of $x$, and $d \lambda$ is the standard Lesbegue
measure in the corresponding open chart of $\mathbb{C}^n$. 

\end{dfn}

In this paper we will denote $N \coloneqq \frac{n(n+1)}{2}$.
Angehrn and Siu proved the existence of a quasi-psh function whose multiplier ideal sheaf has
isolated zero set at $y$ when we pick one point $y \in Y$.
\begin{thm} \cite{AS95}
\label{AS}
Let $Y$ be a smooth projective $n$-dimensional variety, and 
We fix $m \in \mathbb{N}$ such that $m(N+1) L$ is very ample.
We choose a \kah form $\omega_Y$ on $Y$ and a smooth positive metric $h_L$ on $L$ 
such that 
$\sqrt{-1} \Theta _{L , h_L} = \frac{1}{m(N+1)} \omega_{Y} $, where $N = \frac{n(n + 1)}{2}$.
Then for any point $y \in Y$, there exist a quasi-psh function $\varphi$ with neat analytic singularities on $Y$ and a positive number $0 < \varepsilon_0 < 1$, such that 

\begin{enumerate}
\item 
$\sqrt{-1} \Theta _{ L^{\otimes N+1}  h_{L}^{N+1} } +  \deldel \varphi \geq 
\frac{1 - \varepsilon_0}{m(N+1)}
 \omega_{Y} $
\item 
\label{2}
$y$ is an isolated point in the zero variety $V( \mathcal{J}(e^{-\varphi}) )$.

\end{enumerate}
\end{thm}

By the following theorem,
a relative pluricanonical line bundle $K_{X/Y}^{\otimes a}$ has a semipositive singular Hermitian metric which is equal to the fiberwise Bergman kernel metric.

\begin{thm}[\cite{BP08} Theorem 4.2 , \cite{PT14} Collary 4.3.2]
\label{BP}
Let $f \colon \hol{X}{Y} $ be a surjective morphism of smooth projective varieties.
Assume that there exists a regular value $y \in Y$ such that $H^0 \bigl( X_y , K_{X_y}^{\otimes a } \bigr) \not = 0$.
Then the bundle $K_{X/Y}^{\otimes a}$ admits a singular Hermitian metric $h_a$ with semipositive
curvature current such that for any regular value $w \in Y$ and any section 
$s \in H^0 \bigl( X_w , K_{X_w}^{\otimes a } \bigr)$ we have 
$$
| s | _{h_a}^{\frac{2}{a} } (z) \le \int_{X_w} | s | ^{ \frac{2}{a}}
$$
for any $z \in X_w$ up to the identification of $K_{X/Y} |_{X_w} $ with $K_{X_w}$.
We regard $ | s |^{ \frac{2}{a} }$ as a semipositive continuous $( m , m )$ form where $m = \dim X_w$.
\end{thm}

\section{Proof of Theorem \ref{main}}
In this section, we prove Theorem \ref{main}.
\begin{thm}[ = Theorem \ref{main}]
Let $f \colon \hol{X}{Y} $ be a surjective morphism of smooth projective varieties,
with Y of dimension $n$, and $L$ be an ample line bundle on $Y$.
If $y$ is a regular value of $f$, then for any $a \geq 1$ the sheaf 
$$ f_{*}(K_{X}^{\otimes a} ) \otimes L^{\otimes b} 
$$  
is generated by the global sections at $y$ for all $b \geq \frac{n(n-1)}{2} + a(n+1)$.
\end{thm}

Let us first outline the proof.
It is enough to show that
for any regular value $y \in Y$, 
any section 
$s \in H^0 \bigl( X_y , K_{X_y}^{\otimes a } \otimes f^{*}(L )^{\otimes b} | _{X_y} \bigr)$  
can be extended to $X$.
By taking an appropriate singular Hermitian metric on $K_{X}^{\otimes a -1 } \otimes f^{*}(L^{\otimes b})$, 
we can prove there exists a section $S_U$ near $X_y$ such that $ S_U | _{X_y} = s $ by an $L^2$ extension theorem.
To extend $S_U$ to $X$, we solve a $\overline{\partial} $-equation with some weight.

\subsection{ Set up }
We fix a regular value $y \in Y$
and a section 
$s \in H^0 \bigl( X_y , K_{X_y}^{\otimes a } \otimes f^{*}(L )^{\otimes b} | _{X_y} \bigr)$.
We may assume $s \not = 0$.

Let $\omega_X$ be a \kah form on $X$.
We will denote by $h_L$ the smooth positive metric on $L$ and denote by $\omega_{Y}$ the \kah form on $Y$ as in Theorem \ref{AS}.
Since $K_Y \otimes L^{ \otimes n+1 }$ is semiample by Mori theory and Kawamata's basepoint free theorem
(see \cite[Theorem 1.13 and Theorem 3.3]{KM}),
there exists a smooth semipositive metric $h_s$ on $K_Y \otimes L^{ \otimes n+1 }$.
We take 
the singular Hermitian metric $h_{a}$ on $K_{X/Y}^{ \otimes a}$
as in Theorem \ref{BP}.

We will denote by  $\overline{L} \coloneqq K_{X/Y}^{ \otimes(a-1)} \otimes f^{*}(K_Y \otimes L^{ \otimes n+1 }) ^{\otimes a-1} \otimes f^{*} ( L^{ \otimes N+1+\overline{b}} )$ 
and  $\overline{b} \coloneqq b - \frac{n(n-1)}{2} - a(n+1) \geq 0$.
Define $h_{\overline{L}} \coloneqq h_{a}^{ \frac{a-1}{a} }f^{*}( h_{s}^{a-1} h_{L}^{N+1+\overline{b}}  )$,
which is a singular Hermitian metric on $\overline{L}$ with semipositive curvature current.
Note that $K_X \otimes \overline{L} = K_{X}^{\otimes a} \otimes f^{*}(L^{\otimes b}) $.

\subsection{ Local Extension }
We choose a coodinate neighborhood $V$ near $y$ and we set $U \coloneqq f^{-1}( V ) $.
We may regard $V$ as an open ball in $\mathbb{C}^n$ and $y$ as an origin in $\mathbb{C}^n$.
Since $ | s |^{2}_{h_{a}}$ is bounded above on $X_y$ by Theorem \ref{BP}, we obtain
\begin{align}
\begin{split}
\| s \|_{h_{ \overline{L} } , \omega_X }^{2} 
&= 
\int_{X_y} | s | ^{2}_{  h_{ \overline{L} }  , \omega_X} dV_{X_y , \omega_X } \\
&= 
C \int_{X_y} | s |^{ \frac{2(a-1)}{a} }_{h_{a}} |s|^{ \frac{2}{a} } _{\omega_X}  dV_{X_y , \omega_X} \\
&\leq 
C' \int_{X_y} |s|^{ \frac{2}{a} } _{\omega_X}  dV_{X_y , \omega_X} \\
&< +\infty,
\end{split}
\end{align}
where $C$ and $C'$ are some positive constants.
Therefore by the $L^2$ extension theorem
in \cite[Theorem 14.4]{HPS17},
there exists $ S_{U} \in H^{0} \bigl( U , K_{X} \otimes \overline{L} \otimes \mathcal{J}(h_{\overline{L}}) \bigr) $ such that $S_{U} | _{X_y} = s$.

\subsection{ Global Extension }
We denote by $\varphi$ the quasi-psh function on $Y$ as in Theorem \ref{AS} and 
denote by $ \psi \coloneqq \varphi \circ f $.
By Theorem \ref{AS}, we can take a cut-off function $\rho$ near $y$ such that 
\begin{enumerate}
\item ${ \rm { \rm supp } }(\rho) \subset \subset V$,
\item ${ \rm { \rm supp } }(\overline{\partial} \rho) \not \ni y$,
\item $\int_{{ \rm supp }( \overline{\partial} \rho )} e^{- \varphi } dV_{Y , \omega_{Y}} < + \infty $,
\end{enumerate}
and put $ \widetilde {\rho} \coloneqq \rho \circ f$.
We solve the global $\overline{ \partial }$-equation
$
\overline{ \partial } F = \overline{ \partial }( \widetilde{\rho} S_U) 
$
on $X$ with the weight of $ h_{\overline{L}} e^{- \psi}$.

It is easy to check
$\| \widetilde{ \rho } S_U \| ^{2} _{h_{\overline{L}}, \omega_{X} } < + \infty $ and 
$\| \overline{\partial} {(\widetilde{ \rho } S_U) } \|^{2}_{ h_{\overline{L}} , \omega_{X} }  < + \infty $.
Therefore $ \overline{\partial} ( \widetilde{\rho} S_U)  $ gives rise to a cohomology class 
$[  \overline{\partial} ( \widetilde{\rho} S_U) ] $
which is 
$ [\overline{\partial} ( \widetilde{\rho} S_U) ] =0 $
in $ H^{1}\bigl( X , K_X \otimes \overline{L} \otimes \mathcal{J}( h_{\overline{L}} ) \bigr) $.
Since $| S_U |^{2} _{ h_{\overline{L}}} $ is bounded above on $U$ by Theorem \ref{BP} 
( if necessary we take $U$ small enough ), 
we obtain
\begin{align}
\begin{split}
\| \overline{\partial} (\widetilde{ \rho} S_U )  \|^{2}_{ h_{\overline{L}} e^{- \psi}  , \omega_{X} } 
&=
\int _{ U }  | \overline{\partial} (\widetilde{ \rho}   S_U) |^{2} _{ h_{\overline{L}}e^ {- \psi } } dV_{X , \omega_{X} } \\
&\le
C \int_{ f^{-1}( { \rm supp }(\overline{\partial} \rho) )} | S_U |^{2} _{h_{\overline{L}} } e^{ - \psi } dV_{ X , \omega_{X} } \\
&\le
C' \int_{ f^{-1}( { \rm supp }(\overline{\partial} \rho) )} e^{ - \psi } dV_{ X , \omega_{X} } \\
&< + \infty, \\
\end{split}
\end{align}
where $C$ and $C'$ are some positive constants.
Therefore $ \overline{\partial} ( \widetilde{\rho} S_U ) $ is a $\overline{\partial} $-closed $( d , 1 )$ form with $\overline{L}$ value
which is square integrable with the weight of $h_{\overline{L}} e^{- \psi }$, where $d = \dim X $.

We put  $\delta \coloneqq \frac{1- \epsilon_0 }{ 2 (N + \epsilon_0)}$. 
Then we obtain
\begin{align}
\begin{split}
\sqrt{-1} \Theta_{h_{\overline{L}}, h_{\overline{L}}} + (1 + \alpha \delta)  \deldel \psi 
&=
\frac{a-1}{a} \sqrt{-1} \Theta_{ K_{X/Y} , h_{a} }
+ (a-1) f^{*} ( \sqrt{-1}  \Theta_{ K_Y \otimes  L^{ \otimes n+1} , h_s} ) \\
& \,\,\,\,\,\, +
(N+1+\overline{b})f^{*} (\sqrt{-1} \Theta_{ L , h_L} )
+(1+ \alpha \delta) \deldel \psi \\
&\geq 
f^{*} \Bigl( (N+1)\sqrt{-1} \Theta_{ L , h_L}
+(1+ \alpha \delta) \deldel \varphi \Bigr)  \\
&=
f^{*}\Bigl(  
(1+ \alpha \delta)
( \sqrt{-1}\Theta_{ L^{\otimes N+1} , h_{L}^{N+1} } +\deldel \varphi )  
-\alpha \delta \sqrt{-1} \Theta_{ L^{\otimes N+1} , h_{L}^{N+1} } 
\Bigr) \\
&\geq
\frac{ (2 -\alpha) (1- \epsilon_0 ) }{ 2m(N+1)}
f^{*}(\omega_{Y}) \\
&\geq 0
\end{split}
\end{align}
in the sense of current 
for any $\alpha \in [0 , 1]$.
Therefore by the injectivity theorem in \cite[Theorem 1.1]{CDM17}, the natural morphism 
$$
 H^{1}\bigl( X , K_X \otimes \overline{L} \otimes \mathcal{J}( h_{\overline{L}} e^{-\psi} ) \bigr) 
 \rightarrow
 H^{1}\bigl( X , K_X \otimes \overline{L} \otimes \mathcal{J}( h_{\overline{L}} ) \bigr)
 $$
 is injective.
Since
$ [ \overline{\partial} ( \widetilde{\rho} S_U ) ] = 0$  in 
$ H^{1}\bigl( X , K_X \otimes \overline{L} \otimes \mathcal{J}(h_{\overline{L}}  ) \bigr) $,
we obtain 
$ [ \overline{\partial} ( \widetilde{\rho} S_U ) ]  = 0$  in 
$ H^{1}\bigl( X , K_X \otimes \overline{L} \otimes \mathcal{J}( h_{\overline{L}}e^{-\psi} ) \bigr) $.
Hence we obtain a $(d,0)$ form $F$ with $\overline{L}$ value
which is square integrable with the weight of $h_{\overline{L}} e^{- \psi }$ such that
$\overline{ \partial } F = \overline{ \partial }( \widetilde{\rho} S_U) $,
that is we can solve $\overline{\partial } $ equation.

Now we show that $F | _{X_y} \equiv 0$.
To obtain a contradiction,  suppose that $F(x) \not = 0$ for some $x \in X_y$.
We may assume there exists an open set $W$ near $x$ such that $F(z) \not = 0$ for any 
$z \in W$
and $\int_{W} e^{- \psi} dV_{X , \omega_X} = + \infty$ 
since $y$ is an isolated point in the zero variety $V( \mathcal{J}(e^{-\varphi}) )$ by Theorem \ref{AS}.
Since there exists a positive constant $C$ such that
$| F | ^{2} _{ h_{\overline{L}}} \geq C $ on $W$, we have
\begin{align}
\begin{split}
+\infty
>
\| F \|^{2}_{h_{\overline{L}}e^{- \psi} } 
= 
\int_{X} | F | ^{2} _{ h_{\overline{L}}} e^{- \psi} dV_{ X , \omega_X} 
&\geq
\int_{W} | F | ^{2} _{ h_{\overline{L}}} e^{- \psi} dV_{ X , \omega_X} \\
&\geq
C \int_{W} e^{- \psi} dV_{ X , \omega_X} \\
&= + \infty,
\end{split}
\end{align}
which is impossible.

Hence we put $S \coloneqq \widetilde{\rho} S_U -F \in H^{0}\bigl( X , K_X \otimes \overline{L} \bigr) $,
then $S |_{X_y} = (\widetilde{\rho}S_U -F) |_{X_y} = s$,
which completes the proof.

\section{On a log case}
In this section, we prove Theorem \ref{main2}.

\begin{thm}[= Theorem \ref{main2}]
Let $(X , \Delta) $ be a Kawamata log terminal $\mathbb{Q}$-pair of a normal complex analytic variety in Fujiki's class $\mathcal{C}$ and an effective divisor, 
and $Y$ be a smooth projective $n$-dimensional variety.
Let $f \colon \hol{X}{Y} $ be a surjective morphism,
and $L$ be an ample line bundle on $Y$.
For any $a \geq 1$ such that $a(K_{X} + \Delta ) $ is an integral Cartier divisor, the sheaf 
$$ 
f_{*}\Bigl( \mathcal{O}_X \bigl( a(K_{X} + \Delta ) \bigr) \Bigr) \otimes L^{\otimes b} 
$$  
is generated by the global sections at a general point $y \in Y$ either 

\begin{enumerate}
\item for all $b \geq \frac{n(n-1)}{2} + a(n+1) $
, or
\item for all $b \geq \frac{n(n-1)}{2} + 2$ when $K_Y$ is pseudo-effective.
\end{enumerate}
\end{thm}

\begin{proof}
The proof is similar to Theorem \ref{main} and \cite[Theorem C]{Deng17}.
Take a log resolution $\mu \colon \hol {X'}{X}$ of $( X , \Delta)$
we have a compact \kah manifold $X'$
such that $$ K_{X'} = \mu^{*} ( a(K_{X} + \Delta) ) + \sum a\alpha_i E_i -\sum a\beta_j F_j , $$
where $a\alpha_i , a\beta_j \in \mathbb{N}_{+} $ and $\sum_{i} E_i + \sum_{j} F_j$ has simple normal crossing supports.
Since $(X , \Delta) $ is a Kawamata log terminal $\mathbb{Q}$-pair and $\Delta$ is effective,
$E_i$ is an exceptional divisor and $0 < \beta_j <1$.
We denote by $f' \coloneqq f \circ \mu$, which is a surjective morphism between compact \kah manifolds.
Since $E_i$ is an exceptional divisor, the natural morphism
\begin{align}
\begin{split}
H^0 \bigl( X' , \mu^{*} \bigl(  \mathcal{O}_X \bigl( a(K_{X} + \Delta ) \bigr) \otimes f^{*}(L ^{\otimes b} ) \bigr) \bigr)
&\rightarrow
H^0 \bigl( X' ,  \mu^{*} \bigl(  \mathcal{O}_X \bigl( a(K_{X} + \Delta ) \bigr) \otimes f^{*}(L ^{\otimes b} ) \bigr)
\otimes \mathcal{O}_{X'}( \sum a \alpha_i E_i ) \bigr)   \\
&=
H^{0}\bigl( X' , K_{X'}^{\otimes a}  \otimes f'^{*}(L )^{\otimes b} \otimes \mathcal{O}_{X'}( \sum a \beta_j F_j  )  \bigr)\\
\end{split}
\end{align}
is isomorphism. Thus it is enough to show that 
for any general point $y \in Y$, the restriction map  
$$
\pi_{y} \colon
\hol{ H^0\bigl( X' , K_{X'}^{\otimes a} \otimes f^{*}(L^{\otimes b})
 \otimes \mathcal{O}_{X'}( \sum a \beta_j F_j ) \bigr) }
{ H^0 \bigl( X'_y , K_{X'_y}^{\otimes a } \otimes f^{*}(L )^{\otimes b} | _{X'_y}
 \otimes \mathcal{O}_{X'} (\sum a \beta_j F_j ) |_{X'_y} \bigr) }
$$
is surjective.
\\

In case (1),
we choose the canonical singular Hermitian metric $h_F$ on $\mathcal{O}_{X'}(  \sum a\beta_j F_j ) $ 
as in \cite[Example 3.13]{Dem}. We obtain $ \mathcal{J}( h_F^\frac{1}{a} ) = \mathcal{O}_{X'} $ since
 $\sum_{i} E_i + \sum_{j} F_j$ has simple normal crossing supports and $0 < \beta_j <1$.
By \cite[Theorem 3.5]{Cao}, there exists an $a$-th Bergman type metric $h_{a, B}$ on 
$K_{X'/Y} ^{\otimes a} \otimes \mathcal{O}_{X'}( \sum a \beta_j F_j )$. 
We note that 
for any general point $y $ of $f$ such that
$ \mathcal{J}( h_F^\frac{1}{a} |_{X'_y}) = \mathcal{O}_{X'_y}$, and for any section
$s' \in H^{0}( X'_y , K_{X'_y} ^{\otimes a} \otimes \mathcal{O}_{X'}( \sum a \beta_j F_j )|_{X'_y} )$, we have $| s' |^\frac{2}{a}_{ h_{a, B} } \le \int_{X'_y } |s'|^\frac{2}{a} _{h_{F}} < + \infty$
on $X'_y$ by \cite[Theorem 3.5]{Cao}.

We will denote by $\overline{L} \coloneqq K_{X'/Y}^{ \otimes a-1} \otimes \mathcal{O}_{X'}( \sum a \beta_j F_j )  \otimes f^{*}(K_Y \otimes L^{ \otimes n+1 }) ^{\otimes a-1} \otimes f^{*} ( L^{ \otimes \overline{b}} )$ 
and  $\overline{b} \coloneqq b - \frac{n(n-1)}{2} - a(n+1) \geq 0$.
Define a singlar Hermitian metric $h_{\overline{L}} \coloneqq h_{a , B}^{ \frac{a-1}{a} } h_F^\frac{1}{a} f^{*}( h_{s}^{a-1} h_{L}^{\overline{b}}  )$ on $\overline{L}$.
If $y$ is a general point in $Y$ such that
$ \mathcal{J}( h_F^\frac{1}{a} |_{X'_y}) = \mathcal{O}_{X'_y}$, the restriction map $\pi_{y}$ is surjective 
since the same proof works as in Section 3. 
By \cite[section 9.5.D]{Laz},
$\mathcal{J}( h_F^\frac{1}{a} |_{X'_y}) = \mathcal{J}( h_F ^\frac{1}{a}) | _{X'_y} = \mathcal{O}_{X'_y}$ for any general point $y \in Y$. 
Therefore $\pi_{y}$ is surjective for any general point $y \in Y$, which completes the proof.
\\

In case (2),
since $K_Y$ is pseudo-effective, $K_Y ^{\otimes a-1} \otimes L$ is a big line bundle.
Therefore, 
there exists a singular Hermitian metric $h_Y$ on $K_Y ^{\otimes a-1} \otimes L$
with neat analytic singularities 
such that $\sqrt{-1} \Theta_{ K_Y ^{\otimes a-1} \otimes L, h_Y} > 0 $ in the sense of current.

We will denote by $\overline{L} \coloneqq K_{X'/Y}^{ \otimes a-1} \otimes \mathcal{O}_{X'}( \sum a \beta_j F_j )  \otimes f^{*}(K_Y ^{\otimes a-1} \otimes L) \otimes f^{*} ( L^{ \otimes \overline{b}} )$ 
and  $\overline{b} \coloneqq b - N - 2 \geq 0$.
Define a singular Hermitian metric $h_{\overline{L}} \coloneqq h_{a , B}^{ \frac{a-1}{a} } h_F^\frac{1}{a} f^{*}( h_Y h_{L}^{\overline{b}}  )$ on $\overline{L}$.
If $y$ is a general point in $Y$ such that $y \not \in \{ z \in Y \colon h_Y(z) = + \infty \} $ and
$ \mathcal{J}( h_F^\frac{1}{a} |_{X'_y}) = \mathcal{O}_{X'_y}$, the restriction map $\pi_{y}$ is surjective since the same proof works as in Section 3. Since the set $\{ z \in Y \colon h_Y(z) = + \infty \}  $ is Zariski closed, then $\pi_{y}$ is surjective for any general point $y \in Y$, which completes the proof.

\end{proof}

\end{document}